\documentclass{article}

\usepackage[T1]{fontenc}
\usepackage[utf8]{inputenc}
\usepackage{amsmath,amssymb,amsfonts,amsthm}
\usepackage[margin=1in]{geometry}
\usepackage{bm}
\usepackage{mathtools}
\usepackage{microtype}
\usepackage{enumitem}
\usepackage{dsfont}
\usepackage{algorithm}
\usepackage{algorithmic}
\usepackage{authblk}
\usepackage[unicode,psdextra]{hyperref}
\usepackage{bookmark}
\hypersetup{
    colorlinks=true,
    linkcolor=blue,
    citecolor=blue,
    urlcolor=blue,
    pdftitle={Feasibility Correction in Linear Programs},
    pdfauthor={Pedro Abdalla, Roman Vershynin, Guangyi Zou}
}

\theoremstyle{plain}
\newtheorem{theorem}{Theorem}[section]
\newtheorem{lemma}[theorem]{Lemma}
\newtheorem{proposition}[theorem]{Proposition}
\newtheorem{corollary}[theorem]{Corollary}
\theoremstyle{definition}
\newtheorem{definition}[theorem]{Definition}
\newtheorem{problem}[theorem]{Problem}
\theoremstyle{definition}
\newtheorem{remark}[theorem]{Remark}

\newcommand{\R}{\mathbb{R}}
\newcommand{\Sph}{\mathbb{S}}
\newcommand{\one}{\mathbf{1}}
\newcommand{\opnorm}{\mathrm{op}}
\newcommand{\HS}{\mathrm{HS}}

\title{Feasibility Correction in Linear Programs}
\author[1]{Pedro Abdalla\thanks{Email: \href{mailto:pabdalla@wisc.edu}{\texttt{pabdalla@wisc.edu}}.}}
\author[2]{Roman Vershynin\thanks{Email: \href{mailto:rvershyn@uci.edu}{\texttt{rvershyn@uci.edu}}.}}
\author[2]{Guangyi Zou\thanks{Email: \href{mailto:zouguangyi2001@gmail.com}{\texttt{zouguangyi2001@gmail.com}}.}}
\affil[1]{Department of Mathematics, University of Wisconsin--Madison}
\affil[2]{Department of Mathematics, University of California, Irvine, CA 92697}

\date{}

\begin{document}
\maketitle

\begin{abstract}
We study feasibility correction for linear programs. Suppose an initial point satisfies most of the constraints. Can it be moved into the feasible region without changing the objective value? Our main deterministic result says that this is possible if the constraint matrix satisfies a restricted isometry property (RIP). We apply the correction theorem to obtain lower bounds for random linear programs with finite-moment or sub-Weibull entries. We complement these with matching-order upper bounds under
a uniform $p$th-moment bound for some $p>2$ and a delocalization
condition on the objective direction. As an application of our techniques, we also derive lower bounds for the Gaussian width of feasible regions under RIP assumptions.
\end{abstract}


\section{Introduction}

Consider the standard linear program
\begin{equation}
\label{eq:LP}
    z_A^*(b,h)
    :=\sup\bigl\{h^\top x:x\in\R^d,\ A^\top x\le b\bigr\},
\end{equation}
where $A$ is a fixed $d\times n$ matrix, $b\in\R^n$, and $h\in\Sph^{d-1}$. In general, $z_A^{\ast}(b,h)$ does not admit a closed-form formula for arbitrary $A$. A natural question is to estimate $z_A^{\ast}(b,h)$ when $A$ is a random matrix with independent entries. %

\begin{problem}[Optimal value]\label{problem:optimal_value}
Estimate $z^*(b,h)$ when the matrix $A$ is random.
\end{problem}

To solve Problem~\ref{problem:optimal_value}, we consider a conceptually simpler but practically motivated question. The restriction $A^\top x\le b$ in \eqref{eq:LP} consists of $n$ constraints 
\begin{equation}\label{eq: constraints}
a_i^\top x\le b_i
\quad \text{for all } i=1,\ldots,n,
\end{equation}
where $a_1,\ldots,a_n$ denote the columns of $A$.
Imagine that we know a solution $x_0$ that satisfies {\em most} (but not all) of these constraints. Can we efficiently correct this ``almost solution'' to a true solution $x'$ of \eqref{eq:LP} without affecting the objective value?

\begin{problem}[Feasibility correction]\label{problem:correction}
Consider a linear program of the form \eqref{eq:LP}. Given $x_0\in\R^d$ satisfying most constraints in \eqref{eq: constraints}, construct $x'$ such that
\begin{equation}\label{eq:correction}
    A^\top x'\le b
    \quad\text{and}\quad
    h^\top x'=h^\top x_0.
\end{equation}
\end{problem}
The key technical ingredient in our work is a general feasibility correction algorithm for solving Problem~\ref{problem:correction} for any vector $b$, under a general deterministic condition on $A$ that incorporates many random designs of interest. 

Notice that, to preserve the objective value $h^\top x_0$, the correction $x'-x_0$ must lie in $h^\perp$. Thus it is natural that the technical assumption concerns $P_{h^\perp}A$, rather than the matrix $A$ itself. We require that the matrix $P_{h^\perp}A$ satisfy the so-called restricted isometry property (RIP), a standard notion in the field of compressive sensing \cite{foucart2013invitation}.

\begin{definition}[Restricted isometry property]\label{def:RIP}
A matrix $B\in\R^{d\times n}$ satisfies the RIP of order $m$ with constant $\delta_m\in(0,1)$ if%
\[
    (1-\delta_m)\|u\|_2^2
    \le\|Bu\|_2^2
    \le(1+\delta_m)\|u\|_2^2
\]
for every $m$-sparse vector $u\in\R^n$.
\end{definition}

Given an initial input $x_0$, we measure the violations by the violation energy defined as
\begin{equation}\label{eq:energy}
    \sum_{i=1}^n
    \bigl(a_i^\top x_0-b_i\bigr)_+^2.
\end{equation}
Our first main result answers Problem~\ref{problem:correction}: small violation energy against a slightly tightened system guarantees a correction in $h^\perp$.

\begin{theorem}[Feasibility correction]\label{thm:main_results}
Consider the feasibility correction problem in \eqref{eq:correction}. Assume that $P_{h^\perp}A$ satisfies the RIP of order $2m$ with constant $\delta:=\delta_{2m}<1/2$, and let $\varepsilon>0$. If the tightened violation energy of $x_0$ satisfies
\begin{equation}\label{eq:assumption_violation_energy}
    \sum_{i=1}^n
    \bigl(a_i^\top x_0-b_i+2\varepsilon\bigr)_+^2
    \le
    \left(\frac{1-2\delta}{2(1-\delta)}\right)^2
    \varepsilon^2m,
\end{equation}
then there exists a vector $x'$ satisfying all constraints in \eqref{eq: constraints} with 
$$ h^\top x'=h^\top x_0, \qquad\text{and }\|x'-x_0\|_2\le\frac{\varepsilon\sqrt m}{2\sqrt{1-\delta}}.$$
\end{theorem}
The vector $z:=x'-x_0$ can be obtained after finitely many iterations of the correction algorithm in Section~\ref{sec:algorithm}.
\begin{remark}[Satisfying most of the constraints]\label{rmk:violating_condition}
The condition in \eqref{eq:assumption_violation_energy} bounds the violation energy against tightened thresholds. It also implies that most constraints are satisfied:
\[
    |I_0|\varepsilon^2
    \le\sum_{i=1}^n(a_i^\top x_0-b_i+2\varepsilon)_+^2
    \le\varepsilon^2\frac{m}{4}, \qquad\text{where }I_0=\{i:a_i^\top x_0>b_i-\varepsilon\}.
\]
\end{remark}
\paragraph{Random constraints.} As an application of Theorem~\ref{thm:main_results}, we derive sharp estimates, up to absolute constants, for $z^{\ast}(\one,h)$. Informally speaking, if
\begin{enumerate}
    \item $d, n/d\to \infty$; 
    \item $h$ satisfies a delocalization condition $[\log(n/d)]^{1/\alpha}\|h\|_\infty\to0$; and
    \item $A$ has independent centered unit-variance entries with uniformly bounded $\psi_\alpha$ norms,
\end{enumerate}
then
\[
    z^*(\one,h)
    \ge\frac{1-o(1)}{\sqrt{2\log(n/d)}}, \qquad\text{with $1-o(1)$ probability.}
\]
We refer the reader to Theorem~\ref{thm:random_lp_lower} for more details and for the case of finite $p$th moments.

The matching-order upper bound is Theorem~\ref{thm:upper_bound}. In particular, Theorem~\ref{thm:upper_bound} shows that if $d\to\infty$, $n/d\to\infty$, and the delocalization condition $\sqrt{\log(n/d)}\|h\|_\infty\to0$ holds, then
\[
    z^*(\one,h)\le\frac{1+o(1)}{\sqrt{2\log(n/d)}},\qquad\text{with $1-o(1)$ probability.}
\]
We also show an example illustrating that, in general, a delocalization assumption of this form is necessary; see Section~\ref{sec:discussion} for more details.

\paragraph{Related work.}
The work that is closest to ours is due to Bakhshi, Ostrowski, and Tikhomirov \cite{bakhshi2024optimal}, who studied the random linear program in Problem~\ref{problem:optimal_value}, where $b=\one$, $h$ is a fixed unit vector, and $A$ has independent centered unit-variance entries with uniformly bounded sub-Gaussian norms. In the regime $d\to\infty$, $n/d\to\infty$, the authors proved that
\[
    \sqrt{2\log(n/d)}\,z^*(\one,h)\longrightarrow1
    \qquad\text{a.s.}, \text{ if }\log^{3/2}(n/d)\|h\|_\infty\to0.
\]
Their lower bound requires only $\sqrt{\log(n/d)}\|h\|_\infty\to0$. They also obtain sharp spherical mean-width asymptotics in the sub-Gaussian case.

Their proof follows a similar high-level strategy: it starts with an initial point violating a small number of constraints. Then they iteratively correct these violations while preserving the objective value. However, their correction procedure imposes strong assumptions on the tail decay of the entries of $A$.

On the other hand, Theorem~\ref{thm:main_results} bypasses this issue and gives a deterministic RIP condition and an assumption on the violation energy.

In particular, the RIP assumption is satisfied for random $A$ if the entry distributions have finite moments of order $p>4$ or sub-Weibull tails. Under those moment conditions and subject to the delocalization assumptions in Theorem~\ref{thm:random_lp_lower}, the violation energy satisfies \eqref{eq:assumption_violation_energy}. 

For the upper bound, we require only
$\sqrt{\log(n/d)}\|h\|_\infty\to0$ under a uniform
$p$th-moment bound for some $p>2$. Combined with Theorem~\ref{thm:random_lp_lower}, Theorem~\ref{thm:upper_bound} therefore gives the sharp sub-Gaussian limit in probability under this weaker delocalization condition, improving on \cite{bakhshi2024optimal}. In addition, to the best of our knowledge, the connection between the restricted isometry property and the feasibility correction problem is new.

Hoffman's lemma \cite{hoffman1952approximate} bounds the distance between a point $x_0$ and the set
$$\{x \in \mathbb{R}^{d}\mid\forall i,~ a_i^\top x\le b_i
;~ h^\top x=h^\top x_0\}$$ when it is nonempty, in terms of the magnitude of the violation of the constraints $A^\top x_0-b$. In principle, these methods could be applied to \(x_0+h^\perp\), provided that its intersection with the feasible region is nonempty. In contrast, Theorem~\ref{thm:main_results} does not assume that the intersection is nonempty: projected RIP together with a violation energy bound provides an explicit feasible correction. Thus, Theorem~\ref{thm:main_results} applies to random linear programs with independent (potentially heavy-tailed) columns with RIP estimates \cite{adamczak2011restricted,guedon2014restricted,guedon2017interval}. For more applications of Hoffman's lemma, see \cite{leventhal2010randomized} for the convergence of randomized projections and \cite{aharoni1989block,briskman2015block} for block projection methods.

Our results also connect with the literature on convex geometry. Denote the feasible polytope by 
\[K_A:=\{x \in \mathbb{R}^{d}\mid \forall i,~ a_i^\top x\le 1
\}.\]
Classical results provide estimates for the width of the symmetric polyhedron $L_A:=\{x:\|A^\top x\|_\infty\le1\}$, which is contained in the polytope $K_A$ studied here. In fact, the volume-ratio estimate of Carl and Pajor \cite[Theorem~3.3]{carl1988gelfand}, together with Urysohn's inequality, implies that
\[
    w(K_A)\ge w(L_A)\ge c\sqrt{\frac{d}{\log(1+n/d)}},
\]
when $n\ge d$ and $\max_i\|a_i\|_2\le\sqrt d$ (without any RIP assumption). Litvak, Pajor, Rudelson, and Tomczak-Jaegermann \cite{litvak2005smallest} obtain mean-width estimates for symmetric random polyhedra; later work gives containment estimates for $\operatorname{conv}\{\pm a_i:i\le n\}$ under weak moment or small-ball assumptions \cite{guedon2020polytopes,guedon2022geometry}. Our result for Gaussian width in Corollary~\ref{cor:spherical_mean_RIP} complements this literature by presenting a general lower bound under the assumption that $A/\sqrt{d}$ satisfies an appropriate RIP condition. In particular, it recovers optimal estimates for the Gaussian width of $K_A$ when $A$ has independent Gaussian entries.

Smoothed analysis \cite{spielman2004smoothed,dadush2020friendly,huiberts2025upper,bach2025optimal} mainly concerns the average-case complexity of the simplex algorithm, a classical algorithm for solving linear programs. It consists of understanding the running time of the simplex algorithm under Gaussian perturbations of arbitrary normalized data. Our work is about the optimal value of the randomized linear programs. A natural extension of our results would be to consider non-centered random constraints, thereby connecting our setting more closely with the literature on smoothed analysis.

Another interesting connection arises in the spin glass literature. When we write $x=-r\theta$ with $r>0$ and $\|\theta\|_2=1$, the constraints become $\langle a_i,\theta\rangle\ge-1/r$, giving the negative spherical perceptron model. Gardner \cite{gardner1988space} initiated the statistical-mechanics study of perceptron storage capacity. Replica calculations predict full replica symmetry breaking at negative-margin jamming \cite{franz2016simplest}; a numerical determination of the transition under the full-RSB ansatz is given in \cite{annesi2025exact}. For Gaussian constraints at fixed aspect ratio, Montanari, Zhong, and Zhou \cite{montanari2024tractability} analyze random linear optimization programs. We discuss these connections in Section~\ref{sec:discussion}; our random asymptotic results concern $n/d\to\infty$.

\paragraph{Organization.}
Section~\ref{sec:algorithm} presents Algorithm~\ref{alg:orthogonal_proj} for solving Problem~\ref{problem:correction} and proves Theorem~\ref{thm:main_results}. Section~\ref{sec:application} treats applications of our main result to solving Problem~\ref{problem:optimal_value} and establishing Gaussian width estimates for feasible sets under the RIP condition. Section~\ref{sec:discussion} discusses further questions. Appendix~\ref{sec:appendix_tools} collects standard tools for concentration and deviation inequalities.

\paragraph{Notation.}
For $I\subset[n]$, let $A_I$ denote the submatrix of $A$ formed by the columns indexed by $I$. Given $h\in\Sph^{d-1}$, we write $P_{h^\perp}=I-hh^\top$ for the orthogonal projection onto $h^\perp$. The operator and Hilbert--Schmidt norms are denoted by $\|\cdot\|_{\opnorm}$ and $\|\cdot\|_{\HS}$, respectively. For $\alpha>0$, we use the notation
\[
    \|X\|_{\psi_\alpha}
    :=\inf\bigl\{t>0:\mathbb E\exp((|X|/t)^\alpha)\le2\bigr\}.
\]
For $p>0$, we write \(\|X\|_{L^p}=(\mathbb E|X|^p)^{1/p}.\) Constants denoted by $C,c$ may change from line to line; dependencies are displayed as $C(p)$, $c(\varepsilon)$, etc. We write $u_+=\max\{u,0\}$ and $X\lesssim Y$ if $X\le CY$ for an absolute constant $C$.

\section{Correction algorithm under projected RIP}\label{sec:algorithm}

The starting point of this section is Algorithm~\ref{alg:orthogonal_proj}, which gives a finite-step constructive proof of Theorem~\ref{thm:main_results} under a projected RIP condition.
\begin{algorithm}[ht]
\caption{Iterative correction via orthogonal projection}
\label{alg:orthogonal_proj}
\begin{algorithmic}[1]
\STATE \textbf{Input:} $A\in\R^{d\times n}$, $h\in\Sph^{d-1}$, $b\in\R^n$, $x_0\in\R^d$, $m\in\mathbb N$, $\varepsilon>0$.
\STATE $P\gets I-hh^\top$.
\STATE $I_0\gets\{i\in[n]:a_i^\top x_0>b_i-\varepsilon\}$.
\STATE \textbf{if} $I_0=\emptyset$ \textbf{then return} $x_0$.
\STATE $z^{(0)}\gets-PA_{I_0}(A_{I_0}^\top PA_{I_0})^{-1}(A_{I_0}^\top x_0-b_{I_0}+\varepsilon\one)$.
\STATE $\xi^{(0)}\gets z^{(0)}$, \quad $v_0\gets A_{I_0}^\top\xi^{(0)}$.
\FOR{$k=1,2,\ldots$}
    \STATE $\varepsilon_k\gets\|v_{k-1}\|_2/\sqrt m$.
    \STATE $I_k\gets\{i\in[n]:a_i^\top\xi^{(k-1)}\ge\varepsilon_k\}$.
    \STATE \textbf{if} $I_k=\emptyset$ \textbf{then return} $x_0+z^{(k-1)}$.
    \STATE $v_k\gets A_{I_k}^\top\xi^{(k-1)}$.
    \STATE \textbf{if} $\max\{\varepsilon_k,\|v_k\|_2\}\le\varepsilon/2$ \textbf{then return} $x_0+z^{(k-1)}$.
    \STATE $\xi^{(k)}\gets-PA_{I_k}(A_{I_k}^\top PA_{I_k})^{-1}v_k$.
    \STATE $z^{(k)}\gets z^{(k-1)}+\xi^{(k)}$.
\ENDFOR
\end{algorithmic}
\end{algorithm}
Intuitively, given the initial point $x_0$ and $\varepsilon>0$ as in Theorem~\ref{thm:main_results}, we define the initial active set using the tightened thresholds:
\begin{equation}\label{eq:initial_violation_set}
    I_0=\{i:a_i^\top x_0>b_i-\varepsilon\}.
\end{equation}
The goal of Algorithm~\ref{alg:orthogonal_proj} is to find a correction vector $z\in h^\perp$ such that
\begin{equation}\label{problem:stronge}
\begin{split}
    &a_i^\top(x_0+z)\le b_i,\qquad i\in I_0,\\
    &a_i^\top z\le\varepsilon,\qquad i\notin I_0.
\end{split}
\end{equation}
These conditions make $x_0+z$ feasible, while $z\in h^\perp$ preserves the objective value. 

\paragraph{Design principles.}
\begin{enumerate}
\item \textbf{Invariance of the objective value:} To ensure that $h^\top x'=h^\top x_0$, every correction increment $\xi^{(k)}$ must lie in $h^\perp$, the orthogonal complement of $h$.
\item \textbf{Immediate cancellation via batch projection:} At each iteration $k$, the algorithm identifies the active set $I_k$ and performs a joint projection that cancels the preceding residual on $I_k$.
\item \textbf{Controlled residuals via projected RIP:} The RIP ensures that the projected constraint vectors $\{P_{h^\perp}a_i\}$ are nearly orthogonal on sparse supports. This property guarantees geometric decay of the active-set residual vectors.
\end{enumerate}

The initial cumulative correction is
\[
    z^{(0)}:=-P_{h^\perp}A_{I_0}
    (A_{I_0}^\top P_{h^\perp}A_{I_0})^{-1}
    (A_{I_0}^\top x_0-b_{I_0}+\varepsilon\one).
\]
At iteration $k$, we add the correction increment $\xi^{(k)}$ to the cumulative correction $z^{(k-1)}$. Let $I_k$ be the active set and let $v_k$ be the active-set residual vector, where $v_0=A_{I_0}^\top\xi^{(0)}$. As described in Algorithm~\ref{alg:orthogonal_proj}, the update rules for $k\ge1$ are
\begin{equation}\label{eq:update_rules}
\begin{split}
    &\varepsilon_k:=\frac{\|v_{k-1}\|_2}{\sqrt m},\qquad
    I_k:=\{i\in[n]:a_i^\top\xi^{(k-1)}\ge\varepsilon_k\},\qquad
    v_k:=A_{I_k}^\top\xi^{(k-1)},\\
    &\xi^{(k)}:=-P_{h^\perp}A_{I_k}
    (A_{I_k}^\top P_{h^\perp}A_{I_k})^{-1}v_k,\qquad
    z^{(k)}:=z^{(k-1)}+\xi^{(k)}.
\end{split}
\end{equation}
The algorithm stops if $I_k=\emptyset$ or, after computing $v_k$, if $\max\{\varepsilon_k,\|v_k\|_2\}\le\varepsilon/2$. Notice that an empty active set is detected before forming the next residual vector $v_k$. Thus, in each iteration, we have $\varepsilon_{k}\ne 0$.

Every correction increment lies in $h^\perp$, so the objective value is unchanged. The next two results establish the remaining principles.

\begin{lemma}[Second principle]\label{lem:second_principle}
Let $I_k$ be the active set at iteration $k$ of Algorithm~\ref{alg:orthogonal_proj}, as defined in \eqref{eq:update_rules}. Then $I_k$ and $I_{k+1}$ are disjoint whenever both are generated.
\end{lemma}

\begin{proof}
First, $I_0$ and $I_1$ are disjoint. Indeed, for $j\in I_0$, the initialization gives $a_j^\top\xi^{(0)}=-(a_j^\top x_0-b_j+\varepsilon)<0$, whereas membership in $I_1$ requires $a_j^\top\xi^{(0)}\ge\varepsilon_1>0$. For $k\ge1$, suppose $j\in I_k\cap I_{k+1}$. The update rules \eqref{eq:update_rules} give
\[
\begin{aligned}
    A_{I_k}^\top\xi^{(k)}
    &=-\bigl(A_{I_k}^\top P_{h^\perp}A_{I_k}\bigr)
    (A_{I_k}^\top P_{h^\perp}A_{I_k})^{-1}
    A_{I_k}^\top\xi^{(k-1)}\\
    &=-A_{I_k}^\top\xi^{(k-1)}.
\end{aligned}
\]
Consequently, $a_j^\top\xi^{(k)}=-a_j^\top\xi^{(k-1)}\le-\varepsilon_k<0$. However, $j\in I_{k+1}$ implies $a_j^\top\xi^{(k)}\ge\varepsilon_{k+1}>0$, a contradiction. The preceding observation also shows inductively that $\varepsilon_k>0$ whenever a nonempty set $I_k$ is generated.
\end{proof}

The third principle uses projected RIP and the thresholds $\varepsilon_k$ to control active-set sizes and ensure geometric decay of the active-set residual vectors.

\begin{proposition}[Third principle]\label{prop:third_principle}
Let $m\in\mathbb N$, let $I_0$ be the set in \eqref{eq:initial_violation_set}, and assume that $P_{h^\perp}A\in\R^{d\times n}$ satisfies the RIP of order $2m$ with constant $\delta_{2m}<1/2$. Assume that $1\le|I_0|\le m$ and that $x_0$ satisfies the violation-energy condition \eqref{eq:assumption_violation_energy}. Set $v_0=A_{I_0}^\top\xi^{(0)}$ and $v_k=A_{I_k}^\top\xi^{(k-1)}$ for $k\ge1$. Then the following hold for all indices generated by the algorithm:
\begin{enumerate}[label=(\alph*)]
\item \textbf{Cardinality control:} $|I_k|\le m$, so the RIP assumption remains applicable at each iteration.
\item \textbf{Geometric decay:}
\[
    \|v_{k+1}\|_2\le\rho\|v_k\|_2,
    \qquad\rho:=\frac{\delta_{2m}}{1-\delta_{2m}}<1.
\]
\item \textbf{Finite termination:} Algorithm~\ref{alg:orthogonal_proj} stops after at most
\(1+\frac{[\log((1-\rho)\sqrt m)]_+}{\log(1/\rho)}\) loop iterations.
\end{enumerate}
\end{proposition}

We postpone the proof of Proposition~\ref{prop:third_principle} to the end of this section. We now prove the main result of this section.

\begin{proof}[Proof of Theorem~\ref{thm:main_results}]
If $I_0=\emptyset$, then $x_0$ is feasible. Assume below that $I_0\ne\emptyset$ and set $\rho=\delta_{2m}/(1-\delta_{2m})$. By the initialization and \eqref{eq:assumption_violation_energy},
\[
    \frac{\|v_0\|_2}{\sqrt m}
    =\frac{\|A_{I_0}^\top x_0-b_{I_0}+\varepsilon\one\|_2}{\sqrt m}
    \le(1-\rho)\frac{\varepsilon}{2}.
\]
Proposition~\ref{prop:third_principle} gives
\begin{equation}\label{eq:threshold_sum}
    \varepsilon_k\le(1-\rho)\frac{\varepsilon}{2}\rho^{k-1},
    \qquad\sum_{k=1}^\infty\varepsilon_k\le\frac{\varepsilon}{2},
    \qquad\|v_k\|_2\le\rho^k\|v_0\|_2.
\end{equation}
Suppose the algorithm returns at loop iteration $\tau\ge1$, before forming $\xi^{(\tau)}$, and set $K=\tau-1$. It returns $x_0+z^{(K)}$, having computed the correction $z=z^{(K)}$ from the increments $\xi^{(0)},\ldots,\xi^{(K)}$.

The RIP and the update rules give, for every $0\le k\le K$,
\[
    \|\xi^{(k)}\|_2\le\frac{\|v_k\|_2}{\sqrt{1-\delta}}.
\]
Consequently, by \eqref{eq:threshold_sum},
\[
    \|z\|_2
    \le\sum_{k=0}^K\|\xi^{(k)}\|_2
    \le\frac{\|v_0\|_2}{\sqrt{1-\delta}(1-\rho)}
    \le\frac{\varepsilon\sqrt m}{2\sqrt{1-\delta}}.
\]

Fix $i\in[n]$. For every executed iteration $k\in\{1,\ldots,K\}$ with $i\in I_k$, the update rule gives
\begin{equation}\label{eq:pairing_identity}
    a_i^\top(\xi^{(k-1)}+\xi^{(k)})=0.
\end{equation}
Lemma~\ref{lem:second_principle} implies that the pairs $\{k-1,k\}$ appearing in \eqref{eq:pairing_identity} are disjoint. Every nonterminal term not belonging to such a pair satisfies $a_i^\top\xi^{(j)}<\varepsilon_{j+1}$. If $I_{K+1}=\emptyset$, the same bound holds for the terminal term. If the stopping rule based on the threshold and residual norm is met, then the terminal term, when unpaired, satisfies
\[
    a_i^\top\xi^{(K)}
    \le\max\{\varepsilon_{K+1},\|v_{K+1}\|_2\}
    \le\frac{\varepsilon}{2}.
\]

If $i\notin I_0$, the pairing identity and \eqref{eq:threshold_sum} therefore yield
\begin{equation}\label{equ:sum_negligible}
    a_i^\top z
    \le\sum_{j=1}^\infty\varepsilon_j+\frac{\varepsilon}{2}
    \le\varepsilon.
\end{equation}
Since $a_i^\top x_0\le b_i-\varepsilon$, this gives $a_i^\top(x_0+z)\le b_i$.

If $i\in I_0$, then $i\notin I_1$. Hence none of the cancelling pairs contains the initial correction increment $\xi^{(0)}$. Applying the same argument to the sum of the remaining correction increments gives $\sum_{j=1}^K a_i^\top\xi^{(j)}\le\varepsilon$, while the initialization gives $a_i^\top(x_0+\xi^{(0)})=b_i-\varepsilon$. Thus $a_i^\top(x_0+z)\le b_i$. Finally, every correction increment lies in $h^\perp$, so $h^\top z=0$.
\end{proof}

To prove Proposition~\ref{prop:third_principle}, we use the following standard consequence of the RIP.

\begin{lemma}[Approximate orthogonality, {\cite[Lemma~2.1]{MR2412803}}]\label{lem:projected_rip_ortho}
Let $B$ be a matrix that satisfies the RIP of order $m$ with constant $\delta_m>0$. Then
\begin{equation}
    \|B_I^\top B_J\|_{\opnorm}\le\delta_m
\end{equation}
for any disjoint sets $I$ and $J$ such that $|I\cup J|\le m$.
\end{lemma}

\begin{proof}[Proof of Proposition~\ref{prop:third_principle}]
We induct on $k$. Initially, $|I_0|\le m$, and the RIP lower bound makes $z^{(0)}$ well-defined. Assume now that $|I_k|\le m$.

\smallskip
\noindent\textbf{Step 1: Cardinality control.}
Suppose, for contradiction, that $|I_{k+1}|>m$ and choose a subset $J\subset I_{k+1}$ of cardinality exactly $m$. By the update rules \eqref{eq:update_rules},
\[
    \|A_J^\top\xi^{(k)}\|_2
    \ge\sqrt{|J|}\,\varepsilon_{k+1}
    =\sqrt m\,\varepsilon_{k+1}
    =\|v_k\|_2.
\]
On the other hand, by Lemma~\ref{lem:second_principle}, $J$ and $I_k$ are disjoint and satisfy $|J\cup I_k|\le2m$. Thus Lemma~\ref{lem:projected_rip_ortho} and the lower RIP bound in Definition~\ref{def:RIP} give
\[
\begin{aligned}
    \|A_J^\top\xi^{(k)}\|_2
    &\le
    \underbrace{\|A_J^\top P_{h^\perp}A_{I_k}\|_{\opnorm}}_{\le\delta_{2m}}
    \underbrace{\|(A_{I_k}^\top P_{h^\perp}A_{I_k})^{-1}\|_{\opnorm}}_{\le1/(1-\delta_{2m})}
    \|v_k\|_2\\
    &\le\rho\|v_k\|_2<\|v_k\|_2,
\end{aligned}
\]
which is a contradiction because $\delta_{2m}<1/2$. We conclude that $|I_{k+1}|\le m$.

\smallskip
\noindent\textbf{Step 2: Geometric decay.}
By the initialization and the update rule for $\xi^{(k)}$ in \eqref{eq:update_rules}, the next active-set residual vector satisfies
\[
    v_{k+1}
    =\sigma_k(A_{I_{k+1}}^\top P_{h^\perp}A_{I_k})
    (A_{I_k}^\top P_{h^\perp}A_{I_k})^{-1}v_k,
\]
where $\sigma_0=1$ and $\sigma_k=-1$ for $k\ge1$. Since $|I_{k+1}|\le m$ and, by the induction hypothesis, $|I_k|\le m$, we have $|I_k\cup I_{k+1}|\le2m$. Moreover, the sets are disjoint by Lemma~\ref{lem:second_principle}. Thus Lemma~\ref{lem:projected_rip_ortho}, applied to $B=P_{h^\perp}A$, and the lower RIP bound give
\[
\begin{aligned}
    \|v_{k+1}\|_2
    &\le
    \underbrace{\|A_{I_{k+1}}^\top P_{h^\perp}A_{I_k}\|_{\opnorm}}_{\le\delta_{2m}}
    \underbrace{\|(A_{I_k}^\top P_{h^\perp}A_{I_k})^{-1}\|_{\opnorm}}_{\le1/(1-\delta_{2m})}
    \|v_k\|_2\\
    &\le\rho\|v_k\|_2.
\end{aligned}
\]
This proves geometric decay. In particular, $\|v_k\|_2\le\rho^k\|v_0\|_2$ and $\varepsilon_k\le\rho^{k-1}\|v_0\|_2/\sqrt m$. The violation-energy condition gives $\|v_0\|_2/\sqrt m\le(1-\rho)\varepsilon/2$. Thus $\varepsilon_k\le\varepsilon/2$ throughout, and the stopping rule is met as soon as $\|v_k\|_2\le\varepsilon/2$. Geometric decay therefore bounds the number of loop iterations by
\[
    1+\frac{[\log((1-\rho)\sqrt m)]_+}{\log(1/\rho)}.
\]
This proves finite termination. In particular, for a fixed gap $\delta_{2m}\le1/2-\kappa$, the algorithm has polynomial arithmetic complexity.
\end{proof}

\section{Applications}\label{sec:application}
We now provide applications of our general Theorem~\ref{thm:main_results} to random linear programs and to deterministic feasible regions under RIP.

\subsection{Linear programming with random constraints: lower bounds}

Throughout this subsection, $b=\one$, and $A=[a_1,\ldots,a_n]\in\R^{d\times n}$ has independent centered unit-variance entries. To apply Theorem~\ref{thm:main_results}, we need two estimates: a bound on the initial violation energy and a sufficiently large RIP order for the matrix $P_{h^\perp}A/\sqrt d$. We prove these facts separately, using the tools in Appendix~\ref{sec:appendix_tools}.

\begin{proposition}[Expected violation energy]\label{prop:expected_energy}
Let $h\in\Sph^{d-1}$ be deterministic, and fix $\eta>0$ and $\delta\in(0,1)$. There exist constants $C_0>1$ and $C_1>0$, depending only on $\eta,\delta,K$ and on $p$ or $\alpha$ in the corresponding case, such that the following holds. For $n/d\ge C_0$, set
\[
    x_0=\frac{1-\delta}{\sqrt{2(1+\eta)\log(n/d)}}h,
    \qquad\mathcal E=\|(A^\top x_0-\one)_+\|_2^2.
\]
Suppose that one of the following conditions holds.
\begin{enumerate}[label=(\roman*),leftmargin=2em]
\item \textbf{Finite-moment case.} For some $p>4$, we have
\begin{equation}\label{eq:LP_delocalization}
    \|A_{ij}\|_{L^p}\le K\quad\text{for all }i,j,
    \qquad
    \|h\|_p\le C_0^{-1}(n/d)^{-(1+\eta)/p}.
\end{equation}
\item \textbf{Sub-Weibull case.} For some $\alpha\in(0,2]$, we have
\begin{equation}\label{eq:L_inf_delocalization}
    \|A_{ij}\|_{\psi_\alpha}\le K\quad\text{for all }i,j,
    \qquad
    \|h\|_\infty\le C_0^{-1}[\log(n/d)]^{-1/\alpha}.
\end{equation}
\end{enumerate}
Then
\[
    \mathbb E\mathcal E\le C_1d(n/d)^{-\eta}.
\]
\end{proposition}

\begin{proof}
Fix a column $a_i$, and write
\[
    s=\frac{\sqrt{2(1+\eta)\log(n/d)}}{1-\delta},
    \qquad S=a_i^\top h.
\]
Since $x_0=h/s$, the tail-integration formula gives
\[
    \mathbb E(a_i^\top x_0-1)_+^2
    =2\int_1^\infty(t-1)\mathbb P\{S\ge st\}\,dt.
\]
We first bound the probability at $s$, and then control the integral over $t\ge T$, where $T>1$ is a sufficiently large constant. All estimates below are uniform in $i$, and constants depend only on the fixed parameters in the proposition. By increasing $C_0$, we may assume $s\ge1$.

\smallskip
\noindent\textbf{Step 1: The probability at the threshold.}
Set
\[
    V=\left(\sum_{j=1}^d h_j^2A_{ji}^2\right)^{1/2}.
\]
Taking $C_0$ sufficiently large makes $s\|h\|_\infty$ as small as needed under either delocalization assumption: use $\|h\|_\infty\le\|h\|_p$ in the finite-moment case and $\alpha\le2$ in the sub-Weibull case.

Apply Theorem~\ref{lem:self_normalized} to the independent summands $h_jA_{ji}$ at $(1-\delta/2)s$. Their total variance is $1$, and
\[
    \sum_{j=1}^d\mathbb E|h_jA_{ji}|^3
    \le C\|h\|_\infty.
\]
Thus the error in the exponent is at most $Cs^3\|h\|_\infty$. The smallness of $s\|h\|_\infty$ verifies the theorem's two size conditions and makes this error smaller than
$\bigl[(1-\delta/2)^2-(1-\delta)^2\bigr]s^2/2$. Consequently,
\begin{equation}\label{eq:near_self_normalized}
\begin{aligned}
    \mathbb P\{S\ge(1-\delta/2)sV\}
    &\le C\exp\left[-\frac{(1-\delta/2)^2s^2}{2}
                     +Cs^3\|h\|_\infty\right]\le C(n/d)^{-(1+\eta)}.
\end{aligned}
\end{equation}

Lemma~\ref{lem:concentration_V}, using $\|h\|_4^4\le\|h\|_\infty^2$ and, in the finite-moment case, $\|h\|_\infty\le\|h\|_p$, gives
\[
    \mathbb P\{V>1+\delta/2\}
    \le
    \begin{cases}
        C\|h\|_p^p+2e^{-c/\|h\|_p^2},&\text{in the finite-moment case},\\
        2e^{-c/\|h\|_\infty^\alpha},&\text{in the sub-Weibull case}.
    \end{cases}
\]
Each bound is at most $C(n/d)^{-(1+\eta)}$ by the corresponding delocalization assumption. Since $(1-\delta/2)(1+\delta/2)<1$, we conclude that
\[
    \mathbb P\{S\ge s\}
    \le\mathbb P\{S\ge(1-\delta/2)sV\}
       +\mathbb P\{V>1+\delta/2\}
    \le C(n/d)^{-(1+\eta)}.
\]

\smallskip
\noindent\textbf{Step 2: The remaining tail.}
In the finite-moment case, the Fuk--Nagaev inequality in Lemma~\ref{lem:concentration}\ref{item:Fuk_Nagaev} gives
\[
    \mathbb P\{S\ge st\}
    \le CK^p\|h\|_p^p(st)^{-p}+e^{-c_ps^2t^2}.
\]
Integrating and using \eqref{eq:LP_delocalization}, we obtain
\[
    2\int_T^\infty t\mathbb P\{S\ge st\}\,dt
    \le CK^p\|h\|_p^ps^{-p}+Cs^{-2}e^{-c_pT^2s^2}
    \le C(n/d)^{-(1+\eta)}
\]
for sufficiently large $T$.

In the sub-Weibull case, the generalized Bernstein inequality in Lemma~\ref{lem:concentration}\ref{item:Bernstein} gives
\[
    \mathbb P\{S\ge st\}\le2e^{-cs^2t^\alpha},
    \qquad t\ge1.
\]
Indeed, for $\alpha\le1$, \eqref{eq:L_inf_delocalization} gives
$K^\alpha\|h\|_\infty^\alpha\le Cs^{-2}\le Cs^{\alpha-2}$; for $\alpha>1$, interpolation gives
\[
    K^\alpha\|h\|_{\alpha/(\alpha-1)}^\alpha
    \le K^\alpha\|h\|_\infty^{2-\alpha}
    \le Cs^{\alpha-2}.
\]
Both terms in the Bernstein minimum are therefore at least $cs^2t^\alpha$, since $t\ge1$ and $\alpha\le2$. Splitting the exponential into two equal factors and using $s\ge1$, we obtain
\[
    2\int_T^\infty t\mathbb P\{S\ge st\}\,dt
    \le C e^{-cT^\alpha s^2/2}
    \le C(n/d)^{-(1+\eta)}
\]
for sufficiently large $T$.

Finally, monotonicity bounds the integral over $1\le t\le T$ by
$(T-1)^2\mathbb P\{S\ge s\}$. Combining the two parts and summing over columns gives, for sufficiently large $C_1$, 
\[
    \mathbb E\mathcal E
    \le C_1n(n/d)^{-(1+\eta)}
    =C_1d(n/d)^{-\eta}.
\]
This completes the proof.
\end{proof}

We next record the projected RIP estimate needed to apply the correction theorem. 

\begin{lemma}[Projected RIP for random matrices, {\cite[Theorem~1]{guedon2014restricted}}]
\label{prop:random_A_RIP}
Let $A\in\R^{d\times n}$ have independent centered unit-variance entries, and let $h\in\Sph^{d-1}$ be deterministic. Fix $0<\theta<1/4$, and assume that $d\ge1+\theta^{-1}$ and $n\ge d-1$.
\begin{enumerate}[label=(\roman*),leftmargin=2em]
\item \textbf{Finite-moment case.} Suppose that $\|A_{ij}\|_{L^p}\le K$ for all $i,j$ and some $p>4$, and fix $0<\kappa<\min\{1,(p-4)/4\}$. There exist constants $C_0$, $C_1>1$, depending only on $p,K,\kappa$, such that the following holds. If
\[
    2^{C_0/(\theta)}
    \le n\le \left(\frac{\theta}{C_0}\right)^{1+p/2}(d-1)^{p/4},
\]
set
\[
    m=\left\lfloor
    \frac{\theta^{\frac{2p}{p-4-2\kappa}}}{C_1}(d-1)
    \left(\frac{n}{d-1}\right)^{-\frac{4+2\kappa}{p-4-2\kappa}}
    \right\rfloor.
\]
Then $P_{h^\perp}A/\sqrt d$ satisfies the RIP of order $m$ with constant at most $2\theta$, with probability at least
\[
    1-2^{-9}\theta-C_0\theta^{-p/2}n(d-1)^{-p/4}.
\]
\item \textbf{Sub-Weibull case.} Suppose that $\|A_{ij}\|_{\psi_\alpha}\le K$ for all $i,j$ and some $\alpha\in(0,2]$. There exists a constant $C_0>1$, depending only on $\alpha,K$, such that the following holds. If
\[
    \frac{C_0}\theta\le n
    \le \frac{\theta}{C_0}
    \exp\left(\frac12\left(\frac{\theta\sqrt{d-1}}{C_0}\right)^\alpha\right),
\]
set
\[
    m=\left\lfloor
    \frac{\theta^2(d-1)}{C_0}
    \left[\log\left(\frac{C_0n}{\theta^2(d-1)}\right)\right]^{-2/\alpha}
    \right\rfloor.
\]
If $m\ge1$, then $P_{h^\perp}A/\sqrt d$ satisfies the RIP of order $m$ with constant at most $2\theta$, with probability at least
\[
    1-2^{-9}\theta
    -2n\exp\left[-\frac1{C_0}\min\left\{\theta^2(d-1),[\theta(d-1)]^{\alpha/2}\right\}\right].
\]
\end{enumerate}
\end{lemma}

\begin{proof}
Throughout the proof, $C\ge1$ may change from line to line. Write $X_i=P_{h^\perp}a_i$. Rosenthal's inequality and the generalized Bernstein inequality give, in the respective cases,
\[
    \sup_{i,\,u\in h^\perp\cap\Sph^{d-1}}
    \|\langle X_i,u\rangle\|_{L^p}\le C,
    \qquad
    \sup_{i,\,u\in h^\perp\cap\Sph^{d-1}}
    \|\langle X_i,u\rangle\|_{\psi_\alpha}\le C.
\]
We now use the estimate in \cite[Theorem~2.1 and the proof of Theorem~3.1]{guedon2017interval}. For the stated choices of $m$ and ranges of $n$, this estimate shows that $P_{h^\perp}A/\sqrt{d-1}$ satisfies the RIP of order $m$ with constant at most $\theta$, except on an event of probability at most
\[
    2^{-9}\theta+
    \mathbb P\left\{
        \max_i\left|\frac{\|X_i\|_2^2}{d-1}-1\right|>\frac\theta2
    \right\}.
\]
Thus it remains to control $\|X_i\|_2$.

In the finite-moment case,
\[
    \|X_i\|_2^2-(d-1)
    =\sum_{j=1}^d(A_{ji}^2-1)-\bigl((a_i^\top h)^2-1\bigr).
\]
Rosenthal's inequality bounds the $L^{p/2}$ norm of the sum by $C\sqrt d$ and the $L^p$ norm of $a_i^\top h$ by $C$. Hence
\[
    \bigl\|\|X_i\|_2^2-(d-1)\bigr\|_{L^{p/2}}
    \le C\sqrt{d-1}.
\]
Markov's inequality and a union bound give
\[
    \mathbb P\left\{
        \max_i\left|\frac{\|X_i\|_2^2}{d-1}-1\right|>\frac\theta2
    \right\}
    \le C\theta^{-p/2}n(d-1)^{-p/4}.
\]

In the sub-Weibull case, write $\|X_i\|_2^2=a_i^\top P_{h^\perp}a_i$. The projection has operator norm $1$ and squared Hilbert--Schmidt norm $d-1$, so Lemma~\ref{lem:concentration}\ref{item:Hanson_Wright} and a union bound give
\[
    \mathbb P\left\{
        \max_i\left|\frac{\|X_i\|_2^2}{d-1}-1\right|>\frac\theta2
    \right\}
    \le2n\exp\left[-\frac1C\min\left\{\theta^2(d-1),[\theta(d-1)]^{\alpha/2}\right\}\right].
\]

Choosing $C_0$ (and $C_1$ in case~(i)) sufficiently large gives the stated probabilities with normalization $\sqrt{d-1}$. Replacing $\sqrt{d-1}$ by $\sqrt d$ multiplies squared norms by $(d-1)/d\ge1-\theta$, so the lower RIP bound is at least $(1-\theta)^2\ge1-2\theta$, while the upper bound is at most $1+\theta$.
\end{proof}

We can now compare the expected violation energy with the available projected RIP order and apply the correction theorem.

\begin{theorem}[Lower bound for random linear programs]\label{thm:random_lp_lower}
Let $A\in\R^{d\times n}$ have independent centered unit-variance entries, and let $h\in\Sph^{d-1}$ be deterministic. Fix $\eta>0$ and $q\in(0,1)$. There exist constants $C_0,C_1>1$, depending only on $\eta,K,q$ and on $p$ or $\alpha$ in the corresponding case, such that the following holds for $d\ge C_1$ and $n\ge C_1d$. Suppose that one of the following conditions holds.
\begin{enumerate}[label=(\roman*),leftmargin=2em]
\item\label{item:random_lp_lower_finite} \textbf{Finite-moment case.} Suppose that $\|A_{ij}\|_{L^p}\le K$ for all $i,j$ and some $p>4$, with $\eta>4/(p-4)$, and assume that
\[
    \|h\|_p\le C_0^{-1}(n/d)^{-(1+\eta)/p},
    \qquad n\le C_1^{-1}d^{1+1/\eta}.
\]
\item\label{item:random_lp_lower_weibull} \textbf{Sub-Weibull case.} Suppose that $\|A_{ij}\|_{\psi_\alpha}\le K$ for all $i,j$ and some $\alpha\in(0,2]$, and assume that
\[
    \|h\|_\infty\le C_0^{-1}[\log(n/d)]^{-1/\alpha},
    \qquad n\le\exp(C_1^{-1}d^{\alpha/2}).
\]
\end{enumerate}
Then, with probability at least $1-q$,
\[
    z^*(\one,h)\ge\frac{1}{\sqrt{2(1+\eta)\log(n/d)}}.
\]
\end{theorem}

\begin{proof}
Throughout the proof, $C\ge1$ depends only on the parameters in the theorem and may change from line to line.

\smallskip
\noindent\textbf{Step 1: Projected RIP.}
Apply Lemma~\ref{prop:random_A_RIP} with $\theta=q/8$, and let $m$ be half its guaranteed RIP order, rounded down. In case~\ref{item:random_lp_lower_finite}, we choose the auxiliary parameter in the lemma so that, after increasing $C_1$,
\[
    m\ge d(n/d)^{-\frac{\eta+4/(p-4)}{2}}.
\]
This is possible because $\eta>4/(p-4)$. In case~\ref{item:random_lp_lower_weibull}, the bounds $d\ge C_1$ and $\log n\le C_1^{-1}d^{\alpha/2}$ give
\[
    m\ge C^{-1}d[\log(Cn/d)]^{-2/\alpha}.
\]
For sufficiently large $C_1$, the lemma's guaranteed order is at least $4$ and its failure probability is at most $q/2$. Thus, with probability at least $1-q/2$, the matrix $P_{h^\perp}A/\sqrt d$ satisfies the RIP of order $2m$ with constant at most $q/4<1/4$.

\smallskip
\noindent\textbf{Step 2: Initial violation energy.}
Set
\[
    x_0=\frac{1}{\sqrt{2(1+\eta)\log(n/d)}}h,
    \qquad
    S=\sum_{i=1}^n(a_i^\top x_0-1+C_0^{-1})_+^2.
\]
Choose $0<\eta_0<\eta$, with
\[
    \eta_0>\frac{\eta+4/(p-4)}{2}
\]
in case~\ref{item:random_lp_lower_finite}. In Proposition~\ref{prop:expected_energy}, use $\eta_0$ in place of $\eta$ and a sufficiently small fixed contraction parameter. Since $\eta_0<\eta$, we can then increase $C_0$ so that $x_0/(1-C_0^{-1})$ is a contraction toward the origin of the initial point in that proposition. Increasing $C_0$ and $C_1$ also ensures its delocalization and size conditions. Rescaling the threshold and applying the proposition, we obtain
\[
    \mathbb ES\le Cd(n/d)^{-\eta_0}.
\]
Thus
\begin{equation}\label{eq:S/m}
    \frac{\mathbb ES}{m}\longrightarrow0
    \qquad\text{as }n/d\to\infty,
\end{equation}
uniformly over the stated range.

\smallskip
\noindent\textbf{Step 3: Applying the correction theorem.}
Apply Theorem~\ref{thm:main_results} with matrix $A/\sqrt d$, threshold $b=\one/\sqrt d$, and tolerance $\varepsilon=1/(2C_0\sqrt d)$. Since the projected RIP constant is at most $1/4$, the violation-energy condition is implied by
\[
    S\le\frac{m}{36C_0^2}.
\]
By \eqref{eq:S/m} and Markov's inequality, this condition is satisfied with probability at least $1-q/2$ when $d,n/d\ge C_1$, after increasing $C_1$.

A union bound shows that projected RIP and the violation-energy condition hold simultaneously with probability at least $1-q$. On this event, Theorem~\ref{thm:main_results} gives a feasible
point $x'$ with
\[
    h^\top x'=h^\top x_0
    =\frac{1}{\sqrt{2(1+\eta)\log(n/d)}},
\]
which proves the claim.
\end{proof}

\begin{remark}
Sharper failure probabilities for empirical covariance and RIP estimates \cite{tikhomirov2018sample,adamczak2011restricted} may be useful for almost-sure variants of our results. We do not pursue this question here.
\end{remark}

\subsection{Linear programming with random constraints: upper bounds}
\label{sec:random_lp_upper}

We extend the upper bound in \cite[Section~4]{bakhshi2024optimal} to entries with uniformly
bounded $p$th moments for some $p>2$ and weaken the
delocalization condition on $h$. The proof combines a one-sided tail estimate with a count of points determined by active constraints.
\begin{theorem}[Upper bound for random linear programs]
\label{thm:upper_bound}
Assume $A\in\R^{d\times n}$ has independent centered unit-variance entries with bounded $p$th moments, $\|A_{ij}\|_{L^p}\le K$, for some $p>2$. Let $h\in\Sph^{d-1}$ be deterministic. For every $\varepsilon\in(0,1)$, there are constants $C_0,C_1>1$, depending only on $\varepsilon,K, p$, such that
\[
    \sqrt{\log(n/d)}\,\|h\|_\infty\le C_0^{-1},
    \qquad d\ge C_0,\qquad n/d\ge C_0
\]
imply
\begin{equation}\label{eq:random_lp_upper_exponential}
    \mathbb P\left\{
        \sqrt{2\log(n/d)}\,z^*(\one,h)>1+\varepsilon
    \right\}
    \le \exp\left[-C_1^{-1}d(n/d)^{\varepsilon/4}\right].
\end{equation}
\end{theorem}

For $\lambda>0$, consider the hyperplane $\{x:h^\top x=\lambda\}$. We first obtain a lower bound on the probability of violating one constraint at one point. Then we show that, if the hyperplane contains a feasible point, its minimum-norm feasible point is determined by at most $d-1$ active constraints.

Related lower bounds on tail probabilities of subgaussian variables appear in \cite[Propositions~2.6 and~4.2]{bakhshi2024optimal}, where a change of measure is used in the proof of Proposition~2.6. 
\begin{proposition}[One-sided tail bound]
\label{prop:tail_lower_bound_refined}
Let $p>2$, and let $\xi=(\xi_1,\ldots,\xi_d)$ have independent centered unit-variance coordinates with $\|\xi_i\|_{L^p}\le K$ for all $i$. For every $\varepsilon\in(0,1)$, there is a constant $C_0>1$, depending only on $\varepsilon,p,K$, such that
\begin{equation}\label{eq:deloc_cond}
    t\ge C_0,\qquad h\in\Sph^{d-1},\qquad
    t\|h\|_\infty\le C_0^{-1}
\end{equation}
imply
\begin{equation}\label{eq:upper_slice_tail}
    \mathbb P\{y^\top\xi\ge t\}
    \ge \exp\left[-(1+\varepsilon)\frac{t^2}{2}\right]
    \qquad\text{for every }y\in\R^d\text{ with }h^\top y=1.
\end{equation}
\end{proposition}

\begin{proof}[Proof of Proposition~\ref{prop:tail_lower_bound_refined}]
Set $\delta=\varepsilon/16$. Throughout the proof, $C\ge1$ depends only on $p,K$ and may change from line to line. Write $\mathbb P$ for the original distribution of $\xi$. We first construct a product distribution $\mathbb Q$ with mean $th$ and then use H\"older's inequality to return to $\mathbb P$.

\smallskip
\noindent\textbf{Step 1: Shift the mean to $th$.}
Let
\[
    L=\left(\frac{K^p}{\delta}\right)^{1/(p-2)},
    \qquad
    U_i=\frac{\xi_i\mathbf1_{\{|\xi_i|\le L\}}
        -\mathbb E[\xi_i\mathbf1_{\{|\xi_i|\le L\}}]}
        {\mathbb E[\xi_i^2\mathbf1_{\{|\xi_i|\le L\}}]}.
\]
The denominator is at least $1-\delta$, since
\[
    \mathbb E[\xi_i^2\mathbf1_{\{|\xi_i|>L\}}]
    \le K^pL^{2-p}=\delta.
\]
Thus
\begin{equation}\label{eq:upper_bounded_scores}
    \mathbb EU_i=0,\qquad
    \mathbb E[\xi_iU_i]=1,\qquad
    \mathbb EU_i^2\le\frac1{1-\delta},\qquad
    |U_i|\le\frac{2L}{1-\delta}.
\end{equation}
Choose $C_0$ large enough that $t\|h\|_\infty\le C_0^{-1}$ implies $|th_iU_i|\le\delta$ and $|th_i|\le\sqrt\delta$. Define
\begin{equation}\label{eq:upper_change_of_measure}
    F=\frac{d\mathbb Q}{d\mathbb P}
    :=\prod_{i=1}^d(1+th_iU_i).
\end{equation}
Each factor is positive and has expectation one, so $\mathbb Q$ is a product probability distribution. Moreover,
\[
    \mathbb E_{\mathbb Q}\xi_i
    =\mathbb E[\xi_i(1+th_iU_i)]=th_i.
\]

\smallskip
\noindent\textbf{Step 2: Bound the tail under $\mathbb Q$.}
Each coordinate density lies between $1-\delta$ and $1+\delta$. Together with $|th_i|^2\le\delta$, this gives
\[
    \tfrac12\le\operatorname{Var}_{\mathbb Q}(\xi_i)\le2,
    \qquad
    \mathbb E_{\mathbb Q}|\xi_i-th_i|^p\le C.
\]
Fix $y$ with $h^\top y=1$ and write $Z=y^\top\xi-t$. Independence and Rosenthal's inequality~\cite[Eq.~(7)]{pinelis2015exact} give
\[
    \mathbb E_{\mathbb Q}Z^2\ge\tfrac12\|y\|_2^2,
    \qquad
    \mathbb E_{\mathbb Q}|Z|^p\le C\|y\|_2^p,
\]
where the second bound uses $\|y\|_p\le\|y\|_2$. The one-sided moment inequality in \cite[Remark~2.4]{veraar2008lower} therefore gives
\begin{equation}\label{eq:upper_biased_tail}
    \mathbb Q\{y^\top\xi\ge t\}\ge C_1^{-1},
\end{equation}
where $C_1>1$ depends only on $p,K$.

\smallskip
\noindent\textbf{Step 3: Return to $\mathbb P$.}
For $|u|\le\delta$, Taylor's inequality gives
\[
    (1+u)^{1+\delta}
    \le1+(1+\delta)u+
    \frac{\delta(1+\delta)}{2(1-\delta)}u^2.
\]
Using independence, \eqref{eq:upper_bounded_scores}, and $\|h\|_2=1$, we obtain
\[
    \mathbb E_{\mathbb P}F^{1+\delta}
    =\prod_i\mathbb E(1+th_iU_i)^{1+\delta}
    \le\exp\left[\frac{\delta(1+\delta)}{2(1-\delta)^2}t^2\right].
\]
H\"older's inequality gives
\[
    \mathbb Q\{y^\top\xi\ge t\}
    \le\bigl(\mathbb E_{\mathbb P}F^{1+\delta}\bigr)^{1/(1+\delta)}
       \mathbb P\{y^\top\xi\ge t\}^{\delta/(1+\delta)}.
\]
Combining these bounds with \eqref{eq:upper_biased_tail}, we conclude that
\[
\begin{aligned}
    \mathbb P\{y^\top\xi\ge t\}
    &\ge C_1^{-(1+\delta)/\delta}
       \exp\left[-\frac{1+\delta}{(1-\delta)^2}\frac{t^2}{2}\right]\\
    &\ge C_1^{-(1+\delta)/\delta}
       \exp\left[-(1+\varepsilon/2)\frac{t^2}{2}\right]\ge\exp\left[-(1+\varepsilon)\frac{t^2}{2}\right].
\end{aligned}
\]
The second inequality uses $\delta=\varepsilon/16$, and the last holds for $t\ge C_0$ after increasing $C_0$. All choices depend only on $\varepsilon,p,K$, so the bound is uniform over all $y$ with $h^\top y=1$.
\end{proof}
The delocalization condition concerns $h$; the norm $\|y\|_2$ may be arbitrarily large. The next lemma counts the candidate points on the hyperplane. It is a sample-compression argument; see \cite{margellos2015compression} for its connection with random convex optimization.

\begin{lemma}[Counting active constraints]
\label{lem:active_set_slice_bound}
Let $n\ge d\ge2$, and let $a_1,\ldots,a_n\in\R^d$ be independent. Fix $h\in\Sph^{d-1}$ and $\lambda>0$, and write
\[
    K_A=\{x\in\R^d:a_i^\top x\le1\text{ for all }i\},
    \qquad H=\{x\in\R^d:h^\top x=\lambda\}.
\]
Suppose that, for some $q\in(0,1)$,
\begin{equation}\label{eq:slice_uniform_violation}
    \mathbb P\{a_i^\top x>1\}\ge q
    \qquad\text{for every deterministic }x\in H\text{ and every }i.
\end{equation}
Then
\begin{equation}\label{eq:active_set_slice_probability}
\begin{aligned}
    \mathbb P\{K_A\cap H\ne\varnothing\}
    &\le \sum_{k=0}^{d-1}\binom nk(1-q)^{n-k}\le \left(\frac{en}{d-1}\right)^{d-1}e^{-q(n-d+1)}.
\end{aligned}
\end{equation}
\end{lemma}

\begin{proof}
\textbf{Step 1: Define the candidate points.}
For each $I\subset[n]$ with $|I|\le d-1$, let $x_I$ be the minimum-norm point of
\begin{equation}\label{eq:active_set_candidate}
    \{x\in H:a_i^\top x=1\text{ for all }i\in I\}.
\end{equation}
If this set is empty, set $x_I=\lambda h$. Thus $x_I$ always lies in $H$ and depends only on the columns indexed by $I$.

\smallskip
\noindent\textbf{Step 2: Find a feasible candidate.}
Suppose $K_A\cap H$ is nonempty. As a closed convex set, it has a unique minimum-norm point $\widehat x$. Let $J$ index the constraints active at $\widehat x$, so $a_j^\top\widehat x=1$ for $j\in J$. If $J=\emptyset$, then of course $\widehat{x}=\lambda h$.

Now assume $J\ne \emptyset$. If $u\in h^\perp$ satisfies $a_j^\top u=0$ for all $j\in J$, then $\widehat x\pm su$ remain feasible for sufficiently small $s>0$: the active equalities are preserved, and the other constraints have strictly positive slack. Minimality therefore gives
\[
    \widehat x^\top u=0.
\]
Choose $I\subset J$ so that $\{P_{h^\perp}a_i:i\in I\}$ is a basis for $\operatorname{span}\{P_{h^\perp}a_j:j\in J\}$. Then $|I|\le d-1$. Every point in \eqref{eq:active_set_candidate} has the form $\widehat x+u$, where $u$ satisfies the same conditions as above. Hence
\[
    \|\widehat x+u\|_2^2
    =\|\widehat x\|_2^2+\|u\|_2^2
    \ge\|\widehat x\|_2^2.
\]
Thus $\widehat x=x_I$, so at least one candidate is feasible. 

\smallskip
\noindent\textbf{Step 3: Condition and count.}
For a fixed $I$, condition on the columns $(a_i)_{i\in I}$. The point $x_I$ is then fixed, while the remaining columns are independent. By \eqref{eq:slice_uniform_violation},
\[
    \mathbb P\left\{a_j^\top x_I\le1\text{ for all }j\notin I
        \,\middle|\,(a_i)_{i\in I}\right\}
    \le(1-q)^{n-|I|}.
\]
A feasible candidate must satisfy all these remaining constraints. Taking a union bound over $I$ gives the first inequality in \eqref{eq:active_set_slice_probability}. The second follows from
\[
    (1-q)^{n-k}\le e^{-q(n-d+1)}
    \quad(0\le k\le d-1),
    \qquad
    \sum_{k=0}^{d-1}\binom nk\le\left(\frac{en}{d-1}\right)^{d-1}.
\]
\end{proof}

\begin{proof}[Proof of Theorem~\ref{thm:upper_bound}]
\textbf{Step 1: Bound one violation probability.}
Set
\[
    \lambda=\frac{1+\varepsilon}{\sqrt{2\log(n/d)}},
    \qquad t=\sqrt{\frac{2\log(n/d)}{1+\varepsilon}},
\]
and let $H=\{x\in\R^d:h^\top x=\lambda\}$. For each fixed $x\in H$, we have $h^\top(x/\lambda)=1$. Proposition~\ref{prop:tail_lower_bound_refined}, with error parameter $\varepsilon/2$, applies when $C_0$ is sufficiently large. Since $\lambda t=\sqrt{1+\varepsilon}>1$, it gives
\[
\begin{aligned}
    \mathbb P\{a_i^\top x>1\}
    &\ge \mathbb P\{a_i^\top(x/\lambda)\ge t\}\\
    &\ge (n/d)^{-(1+\varepsilon/2)/(1+\varepsilon)}
     \ge (n/d)^{-1+\varepsilon/4}
\end{aligned}
\]
for every $i$, where the last inequality uses $0<\varepsilon<1$.

\smallskip
\noindent\textbf{Step 2: Bound the feasibility probability.}
Applying Lemma~\ref{lem:active_set_slice_bound} with the preceding violation bound gives
\[
\begin{aligned}
    \mathbb P\{K_A\cap H\ne\varnothing\}
    &\le\exp\left[(d-1)\log\frac{en}{d-1}
                 -(n-d+1)(n/d)^{-1+\varepsilon/4}\right]\\
    &\le\exp\left[d\log(2en/d)-\tfrac12d(n/d)^{\varepsilon/4}\right]\\
    &\le\exp\left[-\tfrac14d(n/d)^{\varepsilon/4}\right].
\end{aligned}
\]
The second inequality uses $d\ge2$ and $n\ge2d$. The last follows by taking $C_0$ large enough that $\log(2en/d)\le\frac14(n/d)^{\varepsilon/4}$ whenever $n/d\ge C_0$.

\smallskip
\noindent\textbf{Step 3: Objective value.}
If $z^*(\one,h)>\lambda$, some feasible point has objective value greater than $\lambda$. Since $0\in K_A$ and $K_A$ is convex, scaling that point toward the origin gives a feasible point in $H$. Thus
\[
    \{z^*(\one,h)>\lambda\}
    \subseteq\{K_A\cap H\ne\varnothing\}.
\]
The preceding bound proves \eqref{eq:random_lp_upper_exponential} with $C_1\ge4$.
\end{proof}

\subsection{Gaussian width under RIP}

Recall that the Gaussian width of $T$ is
\[
    w(T)=\mathbb E_{g\sim N(0,I_d)}\sup_{x\in T}g^\top x;
\]
see~\cite{vershynin2018high}. We keep $A$ fixed and apply Theorem~\ref{thm:main_results} to a uniformly random objective direction to bound $w(K_A)$ from below.

\begin{corollary}[Gaussian width under RIP]\label{cor:spherical_mean_RIP}
Let $m\in\mathbb N$ and $A\in\R^{d\times n}$, and suppose that $A/\sqrt d$ satisfies the RIP of order $2m$ with constant $\delta<1/2$.
For every $\varepsilon\in(0,1)$, there is a constant $C>1$, depending only on $\varepsilon$ and $\delta$, such that, whenever $n\ge Cd$ and $d\ge C$,
\[
    w(K_A)\ge
    (1-\varepsilon)\sqrt{
    \frac{d}{2(1+\delta)\Big(
    \log(n/d)
    +\left[\log\left(\frac{Cd}{m\log(n/d)}\right)\right]_+
    \Big)}}.
\]
\end{corollary}

\begin{remark}
For a matrix $A$ with i.i.d. standard Gaussian entries, as $d\to\infty$ and $n/d\to\infty$, \cite[Corollary~1.5]{bakhshi2024optimal} yields
\[
    w(K_A)=(1+o(1))\sqrt{\frac{d}{2\log(n/d)}}
    \qquad\text{almost surely}.
\]
Their spherical mean width equals $2w(K_A)/\mathbb E\|g\|_2$, where $g\sim N(0,I_d)$ and $\mathbb E\|g\|_2\sim\sqrt d$.
If the assumptions of Corollary~\ref{cor:spherical_mean_RIP} hold for every fixed $\varepsilon\in(0,1)$, $\delta=o(1)$, and the additional logarithmic term is $o(\log(n/d))$, then letting $\varepsilon\downarrow0$ recovers the same leading constant in the lower bound.
\end{remark}

We first show that projection onto a random hyperplane preserves the RIP.

\begin{lemma}[RIP under random projection]\label{lem:RIP_random_project}
Let $m\in\mathbb N$ with $2m\le n$, and suppose that $A\in\R^{d\times n}$ satisfies the RIP of order $2m$ with constant $\delta<1$.
Let $\theta$ be uniform on $\Sph^{d-1}$.
For every $0<\gamma<1-\delta$, the matrix $P_{\theta^\perp}A$ satisfies the RIP of order $2m$ with constant at most $\delta+\gamma$, with probability at least
\[
    1-\exp\left(-\frac{\gamma d}{4}
    +2m\log\frac{en}{m}\right).
\]
\end{lemma}

\begin{proof}
Write $\theta=g/\|g\|_2$, where $g\sim N(0,I_d)$.
For any fixed subspace $E$ of dimension at most $2m$, independence of the Gaussian length and direction, together with $\mathbb E\|g\|_2^2=d$, gives
\[
    \mathbb E\exp\left(\frac d4\|P_E\theta\|_2^2\right)
    \le\mathbb E\exp\left(\frac14\|P_Eg\|_2^2\right)
    \le2^m.
\]
For the first inequality, condition on $\theta$ and apply Jensen's inequality to $\|g\|_2^2$. The second uses the exponential moment of a sum of at most $2m$ independent squared standard Gaussian variables.
Markov's inequality and a union bound over the spans of all sets of $2m$ columns show that $\|P_E\theta\|_2^2\le\gamma$ for every such span, except on an event of probability at most
\[
    \binom n{2m}2^m e^{-\gamma d/4}
    \le\exp\left(-\frac{\gamma d}{4}
    +2m\log\frac{en}{m}\right).
\]
On the complementary event, $|\theta^\top Au|\le\sqrt\gamma\,\|Au\|_2$ for every $2m$-sparse vector $u$, so
\[
\begin{aligned}
    \|P_{\theta^\perp}Au\|_2^2
    &=\|Au\|_2^2-|\theta^\top Au|^2\\
    &\ge(1-\gamma)(1-\delta)\|u\|_2^2
    \ge(1-\delta-\gamma)\|u\|_2^2.
\end{aligned}
\]
The upper RIP bound follows because orthogonal projection is a contraction.
\end{proof}

\begin{proof}[Proof of Corollary~\ref{cor:spherical_mean_RIP}]
Let $\theta$ be uniform on $\Sph^{d-1}$, set $c=(1-2\delta)/128$, and choose $C$ sufficiently large depending only on $\varepsilon$ and $\delta$.
Set
\[
    t=\log(n/d)
    +\left[\log\left(\frac{Cd}{m\log(n/d)}\right)\right]_+,
    \qquad
    x_0=\frac{1-\varepsilon/2}{\sqrt{2(1+\delta)t}}\theta.
\]
We will find a feasible point with the same objective value as $x_0$ for a proportion at least $1-\varepsilon/4$ of directions.

\smallskip
\noindent\textbf{Step 1: A direct feasible point.}
The RIP gives $\|a_i\|_2^2\le(1+\delta)d$.
For every integer $k\ge1$, the spherical moment formula gives
\[
    \mathbb E(a_i^\top\theta)^{2k}
    =\frac{(2k-1)!!\,\|a_i\|_2^{2k}}{d(d+2)\cdots(d+2k-2)}
    \le(2k-1)!!(1+\delta)^k.
\]
The odd moments vanish, so expansion of the moment-generating function and Chernoff's inequality yield
\[
    \mathbb P\{a_i^\top\theta\ge u\}
    \le\exp\left(-\frac{u^2}{2(1+\delta)}\right),
    \qquad u>0.
\]
If $\log(n/d)>cd/2$, a union bound gives
\[
\begin{aligned}
    \mathbb P\{x_0\notin K_A\}
    &\le n\exp\left(-\frac{t}{(1-\varepsilon/2)^2}\right)\le d(n/d)^{-\varepsilon}
    \le d e^{-\varepsilon cd/2}
    \le\frac{\varepsilon}{4}.
\end{aligned}
\]
Here we used $t\ge\log(n/d)$, $(1-\varepsilon/2)^{-2}\ge1+\varepsilon$, and $d\ge C$.
Thus $x_0$ itself is feasible with the required probability in this case.

\smallskip
\noindent\textbf{Step 2: Correction at a smaller RIP order.}
Suppose now that $\log(n/d)\le cd/2$, and set
\[
    m_0=\min\left\{m,\left\lfloor\frac{cd}{\log(n/d)}\right\rfloor\right\}.
\]
The original RIP still holds at order $2m_0$, and the present case gives
\[
    m_0\ge\frac12\min\left\{m,\frac{cd}{\log(n/d)}\right\}.
\]
Since $n>d$, the original RIP forces $2m\le d$.
For sufficiently large $C$, monotonicity of $u\log(en/u)$ on $(0,n]$ gives
\[
    m_0\log\frac{en}{m_0}
    \le\frac{cd}{\log(n/d)}
    \log\left(\frac{en\log(n/d)}{cd}\right)
    \le2cd.
\]
Apply Lemma~\ref{lem:RIP_random_project} to $A/\sqrt d$ at order $2m_0$ with $\gamma=(1-2\delta)/4$.
The projected matrix satisfies the RIP with constant at most $(1+2\delta)/4$, except on an event of probability at most
\[
    \exp\left[-\frac{(1-2\delta)d}{16}+4cd\right]
    =\exp\left[-\frac{(1-2\delta)d}{32}\right]
    \le\frac{\varepsilon}{8}.
\]
We apply Theorem~\ref{thm:main_results} with matrix $A/\sqrt d$, threshold $\one/\sqrt d$, initial point $x_0$, and tolerance $\varepsilon/(8\sqrt d)$.
After multiplying its violation-energy condition by $d$, the energy to control is
\[
    S=\sum_{i=1}^n\bigl(a_i^\top x_0-1+\varepsilon/4\bigr)_+^2.
\]
Since \(\left(\frac{1-\varepsilon/4}{1-\varepsilon/2}\right)^2\ge1+\frac{\varepsilon}{2},\)
the tail bound from Step~1 and the tail-integration formula give
\[
\begin{aligned}
    \mathbb ES
    &\le2n(1-\varepsilon/4)^2
    \int_1^\infty(u-1)e^{-(1+\varepsilon/2)tu^2}\,du\\
    &\le\frac{n e^{-(1+\varepsilon/2)t}}{t}
    \le(n/d)^{-\varepsilon/2}
    \min\left\{\frac mC,\frac{d}{\log(n/d)}\right\}.
\end{aligned}
\]
Consequently,
\[
    \frac{\mathbb ES}{m_0}
    \le\frac2c(n/d)^{-\varepsilon/2}.
\]
On the projected RIP event, the condition in Theorem~\ref{thm:main_results} is implied by
\[
    S\le\frac{\varepsilon^2m_0}{64}
    \left(\frac{1-2\delta}{3-2\delta}\right)^2.
\]
Markov's inequality makes the failure probability at most $\varepsilon/8$ after increasing $C$.
A union bound and Theorem~\ref{thm:main_results} therefore give a feasible point $x'$ with $\theta^\top x'=\theta^\top x_0$ with probability at least $1-\varepsilon/4$.

\smallskip
\noindent\textbf{Step 3: Averaging over objective directions.}
In both cases, the support function of $K_A$ is at least $\theta^\top x_0$ with probability at least $1-\varepsilon/4$.
It is nonnegative in every direction because $0\in K_A$.
Independence of the length and direction of $g\sim N(0,I_d)$ and Tonelli's theorem give
\[
\begin{aligned}
    w(K_A)
    &=\mathbb E\|g\|_2\,
    \mathbb E_\theta\sup_{x\in K_A}\theta^\top x\ge\frac{(1-\varepsilon/4)(1-\varepsilon/2)\mathbb E\|g\|_2}
    {\sqrt{2(1+\delta)t}}\ge(1-\varepsilon/4)^2(1-\varepsilon/2)
    \sqrt{\frac{d}{2(1+\delta)t}}.
\end{aligned}
\]
The last inequality holds for $d\ge C$, since $\mathbb E\|g\|_2\ge(1-\varepsilon/4)\sqrt d$ for sufficiently large $d$.
Now use $(1-\varepsilon/4)^2(1-\varepsilon/2)\ge1-\varepsilon$ and substitute the definition of $t$.
\end{proof}

\section{Discussion and further questions}\label{sec:discussion}

We discuss extensions and further questions related to our work.

\paragraph{Necessity of the delocalization condition.}
In the proof of Theorem~\ref{thm:upper_bound}, the delocalization condition is used only in the one-sided tail bound in Proposition~\ref{prop:tail_lower_bound_refined}. A Rademacher example shows that the logarithmic exponent in this condition is optimal. Indeed, let $A$ have independent symmetric $\{-1,1\}$-valued entries. Since $a_i^\top x\le\|x\|_1$, the feasible region $K_A$ contains the unit $\ell_1$ ball $B_1^d$. Hence, deterministically,
\[
    z^*(\one,h)\ge\sup_{x\in B_1^d}h^\top x=\|h\|_\infty,
\]
which implies that the asymptotic formula $z^*(\one,h)=(1+o(1))/\sqrt{2\log(n/d)}$ fails whenever $\sqrt{2\log(n/d)}\|h\|_\infty$ is bounded from below by a constant strictly greater than one. This example establishes the sharpness of the logarithmic dependence in the delocalization condition.

\paragraph{Sharp constants under finite moments.}
Theorem~\ref{thm:random_lp_lower} leaves a gap in terms of multiplicative constants between the finite-moment lower bound and the sharp upper bound. The restriction $\eta>4/(p-4)$ arises from comparing the violation energy with the available projected RIP order; it does not imply that the loss in the leading constant is necessary. Let the entries of $A$ be independent copies of a fixed centered unit-variance random variable with a finite $p$th moment for some $p>4$. Under suitable growth and delocalization assumptions, does $\sqrt{2\log(n/d)}\,z^*(\one,h)$ converge to $1$ in probability? One can also ask how far the delocalization assumption can be weakened.

\paragraph{Correction under weaker matrix assumptions.}
Write $B=P_{h^\perp}A$. The proof controls the active-set sizes, the inverse Gram matrices, and the matrices that transfer the active-set residual vectors:
\[
    B_{I_{k+1}}^\top B_{I_k}(B_{I_k}^\top B_{I_k})^{-1}.
\]
The proof would extend if analogous bounds held throughout the iterations of the correction algorithm. Can such bounds hold for useful classes of matrices without uniform projected RIP? For heavy-tailed matrices, can they be established for the active sets selected by the algorithm? The argument must also prevent the active sets from becoming too large; contraction estimates that assume small active sets are not enough.

There are deterministic constructions with a different goal. Under $\max_i\|a_i\|_2\le\sqrt d$ and $n/d\to\infty$, the spherical discrepancy algorithm of Jones and McPartlon~\cite{jones2020spherical} constructs, in polynomial time, a feasible point of norm at least $(1-o(1))/\sqrt{2\log(n/d)}$. It does not preserve a prescribed objective value. This leaves room for assumptions that control both feasibility and the direction of the correction vector.

\paragraph{Non-centered random constraints.}
Let $A=A_0+\sigma G$, where $A_0=[\mu_1,\ldots,\mu_n]$ is deterministic, $\sigma>0$, and $G$ has independent standard Gaussian entries. Drift in the objective direction can change the optimal value. Let $z_G$ denote the optimal value for $G$. If $A_0=\beta h\one^\top$ and $K_G$ is bounded, a calculation gives
\[
    z^*(\one,h)=
    \begin{cases}
        z_G/(\sigma+\beta z_G),&\sigma+\beta z_G>0,\\
        +\infty,&\sigma+\beta z_G\le0.
    \end{cases}
\]
As $d\to\infty$ and $n/d\to\infty$, we have $z_G\sim[2\log(n/d)]^{-1/2}$ in probability, so the centered asymptotic formula remains valid precisely when $\beta=o(\sigma\sqrt{\log(n/d)})$. For a general deterministic mean matrix \(A_0\), which assumptions ensure
\[
    z^*(\one,h)=\frac{1+o(1)}{\sigma\sqrt{2\log(n/d)}}
\]
in probability, uniformly over the admissible $A_0$?

\paragraph{The proportional regime.}
Let $n/d\to\alpha>2$, and suppose that $A$ has independent standard Gaussian entries. Wendel's theorem implies that $K_A$ is bounded with probability $1-o(1)$~\cite{wendel1962problem}. Gaussian comparison already yields asymptotics for the optimal value of ball-constrained linear optimization~\cite[Appendix~B]{montanari2024tractability}. For non-Gaussian entries, can one determine the asymptotic optimal value, identify deterministic assumptions for sharp bounds, and develop an efficient correction algorithm?

\paragraph{Gaussian width and one-sided assumptions.}
Corollary~\ref{cor:spherical_mean_RIP} gives a Gaussian-width lower bound under deterministic RIP assumptions. Deterministic upper bounds, however, require one-sided information. Changing column signs preserves the RIP, but the signs can be chosen so that $\langle a_i,v\rangle\ge0$ for every $i$ and a fixed $v\ne0$. Then $-tv\in K_A$ for all $t\ge0$, so $w(K_A)$ is infinite. What additional assumptions suffice for a Gaussian-width upper bound?

\appendix

\section{Technical tools}\label{sec:appendix_tools}

We collect the tools used in Section~\ref{sec:application}. Throughout, $\Phi$ denotes the standard Gaussian distribution function.

\begin{theorem}[Self-normalized large deviations, {\cite[Theorem~2.1]{MR2016616}}]\label{lem:self_normalized}
Let $Y_1,\ldots,Y_d$ be independent centered random variables with finite variances such that $\sum_{i=1}^d\mathbb EY_i^2>0$. Set
\[
    S_d=\sum_{i=1}^dY_i,
    \qquad B_d=\left(\sum_{i=1}^d\mathbb EY_i^2\right)^{1/2},
    \qquad V_d=\left(\sum_{i=1}^dY_i^2\right)^{1/2}.
\]
For $x\ge0$, define
\begin{equation}
\begin{aligned}
    \Delta_{d,x}
    &=\frac{(1+x)^2}{B_d^2}
    \sum_{i=1}^d\mathbb E\left[Y_i^2\one_{\{|Y_i|>B_d/(1+x)\}}\right]+\frac{(1+x)^3}{B_d^3}
    \sum_{i=1}^d\mathbb E\left[|Y_i|^3\one_{\{|Y_i|\le B_d/(1+x)\}}\right].
\end{aligned}
\end{equation}
There are absolute constants $A>1$ and $0<C\le A$ such that
\begin{equation}
    \frac{\mathbb P\{S_d\ge xV_d\}}{1-\Phi(x)}\le e^{C\Delta_{d,x}}
    \quad\text{and}\quad
    \frac{\mathbb P\{S_d\ge xV_d\}}{1-\Phi(x)}\ge e^{-C\Delta_{d,x}}
\end{equation}
for every $x\ge0$ satisfying $x^2\max_i\mathbb EY_i^2\le B_d^2$ and $\Delta_{d,x}\le(1+x)^2/A$.
\end{theorem}

\begin{lemma}[Concentration inequalities]\label{lem:concentration}
Let $X_1,\ldots,X_n$ be independent centered random variables with finite variances.
\begin{enumerate}[label=(\roman*)]
\item\label{item:Fuk_Nagaev} \textbf{Fuk--Nagaev inequality.} We use the explicit form in \cite[Eq.~(1.7)]{Rio17Fuk}, which follows from the inequalities of Fuk and Nagaev \cite{MR326835,MR542129}.

Let $\sigma^2=\sum_{i=1}^n\mathbb EX_i^2$. For $p>2$, set $C_p=(\sum_{i=1}^n\mathbb EX_{i+}^p)^{1/p}$. Then, for every $x>0$,
\[
    \mathbb P\left\{\sum_{i=1}^nX_i\ge x\right\}
    \le\left(\frac{(p+2)C_p}{px}\right)^p
    +\exp\left(-\frac{2x^2}{(p+2)^2e^p\sigma^2}\right).
\]
\item\label{item:Bernstein} \textbf{Generalized Bernstein inequality} \cite[Theorem~3.1]{kuchibhotla2022moving}.
Suppose that $\|X_i\|_{\psi_\alpha}\le K$ for some $\alpha\in(0,2]$. Set $\alpha^*=\infty$ when $\alpha\le1$ and $\alpha^*=\alpha/(\alpha-1)$ when $\alpha>1$. Then, for every $a\in\R^n$ and $t\ge0$,
\[
    \mathbb P\left\{\left|\sum_{i=1}^n a_iX_i\right|\ge t\right\}
    \le2\exp\left[-\frac1{C_\alpha}\min\left\{
    \frac{t^2}{K^2\|a\|_2^2},
    \left(\frac{t}{K\|a\|_{\alpha^*}}\right)^\alpha\right\}\right].
\]
\item\label{item:Hanson_Wright} \textbf{Generalized Hanson--Wright inequality} \cite[Theorem~2.1]{MR4633263}.
Suppose that $\|X_i\|_{\psi_\alpha}\le K$ for some $\alpha\in(0,2]$. Let $A\in\R^{n\times n}$ be symmetric. Then, for every $t\ge0$,
\[
    \mathbb P\left\{\left|X^\top AX-\mathbb E[X^\top AX]\right|\ge t\right\}
    \le2\exp\left[-\frac1{C_\alpha}\min\left\{
    \frac{t^2}{K^4\|A\|_{\HS}^2},
    \left(\frac{t}{K^2\|A\|_{\opnorm}}\right)^{\alpha/2}\right\}\right].
\]
\end{enumerate}
\end{lemma}

\begin{lemma}[Concentration of a weighted sum of squares]\label{lem:concentration_V}
Let $h\in\Sph^{d-1}$, and let $X_1,\ldots,X_d$ be independent centered unit-variance random variables. Set
\begin{equation}
    W:=\sum_{i=1}^d(h_iX_i)^2.
\end{equation}
\begin{enumerate}[label=(\alph*),leftmargin=2em]
\item \textbf{Finite-moment case.} Assume that, for some $p>4$, $\mathbb E|X_i|^p\le K^p$ for every $i\in[d]$. There are an absolute constant $C$ and a constant $c_p$ such that, for every $t>0$,
\begin{equation}
    \mathbb P\{W-1\ge t\}
    \le\left(\frac{CK^2\|h\|_p^2}{t}\right)^{p/2}
    +2\exp\left(-\frac{c_pt^2}{K^4\|h\|_4^4}\right).
\end{equation}
\item \textbf{Sub-Weibull case.} Assume that $\|X_i\|_{\psi_\alpha}\le K$ for some $\alpha\in(0,2]$. There is a constant $C_\alpha$ such that, for every $t>0$,
\begin{equation}
    \mathbb P\{W-1\ge t\}
    \le2\exp\left[-\frac1{C_\alpha}\min\left\{
    \frac{t^2}{K^4\|h\|_4^4},
    \left(\frac{t}{K^2\|h\|_\infty^2}\right)^{\alpha/2}\right\}\right].
\end{equation}
\item \textbf{Lower deviation.} Assume that $\mathbb E|X_i|^4\le K^4$. Then, for every $t>0$,
\begin{equation}
    \mathbb P\{W-1\le-t\}
    \le\exp\left(-\frac{t^2}{2K^4\|h\|_4^4+\frac23\|h\|_\infty^2t}\right).
\end{equation}
\end{enumerate}
\end{lemma}

\begin{proof}
\textbf{Finite-moment case.} Set $Z_i=h_i^2(X_i^2-1)$. Since $\mathbb E|X_i|^p\le K^p$,
\[
\begin{aligned}
    \|Z_i\|_{L^{p/2}}
    &=h_i^2\|X_i^2-1\|_{L^{p/2}}\le h_i^2\bigl(\|X_i\|_{L^p}^2+1\bigr)
    \le2K^2h_i^2.
\end{aligned}
\]
We apply Lemma~\ref{lem:concentration}\ref{item:Fuk_Nagaev} to $\sum_{i=1}^dZ_i$ with moment order $q=p/2$. Moreover,
\[
    C_q
    \le\left(\sum_{i=1}^d\|Z_i\|_{L^{p/2}}^{p/2}\right)^{2/p}
    \le\left(\sum_{i=1}^d(2K^2h_i^2)^{p/2}\right)^{2/p}
    =2K^2\|h\|_p^2,
\]
and
\[
    \sum_i\mathbb EZ_i^2
    \le\sum_i h_i^4\mathbb EX_i^4
    \le K^4\sum_i h_i^4
    =K^4\|h\|_4^4.
\]
The Fuk--Nagaev inequality gives the claim.

\smallskip
\noindent\textbf{Sub-Weibull case.} Write $W=X^\top AX$, where $A=\operatorname{diag}(h_1^2,\ldots,h_d^2)$. Then $\operatorname{tr}A=1$, $\|A\|_{\HS}=\|h\|_4^2$, and $\|A\|_{\opnorm}=\|h\|_\infty^2$. The claim follows from Lemma~\ref{lem:concentration}\ref{item:Hanson_Wright} applied to $X^\top AX-\mathbb E[X^\top AX]$.

\smallskip
\noindent\textbf{Lower deviation.} Recall that $Z_i=h_i^2(X_i^2-1)$. Since $X_i^2\ge0$, each $Z_i$ is bounded below:
\[
    Z_i\ge-h_i^2\ge-\|h\|_\infty^2.
\]
The total variance satisfies
\[
    \sigma^2=\sum_{i=1}^d\mathbb EZ_i^2
    =\sum_{i=1}^d h_i^4(\mathbb EX_i^4-1)
    \le K^4\sum_{i=1}^d h_i^4
    =K^4\|h\|_4^4.
\]
Applying the one-sided Bernstein inequality to the lower tail (see, e.g., \cite[Section~2.8]{MR3185193}), we obtain
\[
    \mathbb P\{W-1\le-t\}
    \le\exp\left(-\frac{t^2}{2K^4\|h\|_4^4+\frac23\|h\|_\infty^2t}\right).
\]
This completes the proof.
\end{proof}

\section*{Acknowledgments}
Guangyi Zou is the corresponding author.

\section*{Funding}
The authors are supported by NSF Grant DMS 2451011 and U.S. Air Force Grant FA9550-25-1-0294.

\paragraph{Declaration of AI use.}
Generative AI tools were used for mathematical exploration, to help simplify the upper bound argument in Section 3.2, and to identify Lemma 3.7. They were also used to assist in checking routine computations, proofreading, and polishing the exposition throughout the manuscript. The authors take responsibility for all mathematical statements, proofs, and references.

\end{document}